\documentclass[12pt]{article}
\usepackage[french, english]{babel}
\usepackage[T1]{fontenc}
\usepackage[utf8]{inputenc}
\usepackage{amsmath,amssymb,latexsym,amsthm,amscd}
\usepackage[abbrev]{amsrefs}
\usepackage{tikz}
\usetikzlibrary{arrows}
\usetikzlibrary{matrix}
\usetikzlibrary{intersections}
\usetikzlibrary{shapes}
\usepackage{microtype}
\usepackage[margin = 3cm]{geometry}
\usetikzlibrary{calc}

\newtheoremstyle{dotless}{}{}{\itshape}{}{\bfseries}{}{ }{}

\RequirePackage{graphicx}

\theoremstyle{plain}
\newtheorem{theorem}{Theorem}[section]

\newtheorem{corollary}[theorem]{Corollary}
\newtheorem{lemma}[theorem]{Lemma}

\newtheorem*{RAAG conjecture}{Action Dimension Conjecture for RAAGs}
\newtheorem*{Aconj}{Action Dimension Conjecture}

 \theoremstyle{definition}
\newtheorem{definition}[theorem]{Definition} 
\theoremstyle{remark}
\newtheorem*{example}{Example}
\newtheorem*{remark}{Remark}

\newtheorem*{acknowledge}{Acknowledgement}

\newcommand{\geps}{\varepsilon}

\newcommand{\gs}{\sigma}
\newcommand{\gt}{\tau}

\newcommand{\gG}{\Gamma}
\newcommand{\gD}{\Delta}

\newcommand{\diam}{diam}

\newcommand{\Out}{\operatorname{Out}}

\newcommand{\lk}{\operatorname{Lk}}

\newcommand{\vk}{\operatorname{vk}}

\newcommand{\vkt}{\vk_{\zz/2}}

\newcommand{\obdim}{\operatorname{obdim}}

\newcommand{\actdim}{\operatorname{actdim}}

\newcommand{\gdim}{\operatorname{gdim}}

\newcommand{\minus}{^{-1}}

\newcommand{\OL}{{O(L)}}

\newcommand{\Opp}{{Opp}}

\newcommand{\zz}{\mathbb{Z}}

\begin{document}
\title{Action Dimension of Lattices in Euclidean Buildings}
\author{Kevin Schreve}

\maketitle
\abstract{The action dimension of a group $\Gamma$ is the minimal dimension of a contractible manifold that $\Gamma$ acts on properly discontinuously. We show that if $\gG$ acts properly and cocompactly on a thick Euclidean building, then the action dimension is bounded below by twice the dimension of the building. We also compute the action dimension of $S$-arithmetic groups over number fields, partially answering a question of Bestvina, Kapovich, and Kleiner. 

\section{Introduction}
The \emph{action dimension} of a countable discrete group $\Gamma, \actdim(\Gamma)$, is the minimal dimension of contractible manifold $M$ that $\Gamma$ acts on properly discontinuously. If $\Gamma$ is torsion-free, then $M/\gG$ is a model for the classifying space $B\gG$, and hence $\actdim(\gG)$ is the minimal dimension of such a manifold model. If we assume that $\Gamma$ has a finite model for $B\gG$, a theorem of Stallings implies that $\actdim(\gG)$ is bounded above by twice the geometric dimension of $\gG$ \cite{stallings}.

The motivation for this paper is the following conjecture of Davis and Okun \cite{do01}, which claims that nontrivial $L^2$-Betti numbers of a discrete group should provide a lower bound for action dimension. This generalizes an older conjecture of Singer on the vanishing of the $L^2$-cohomology of closed aspherical manifolds.

\begin{Aconj}
If $\gG$ has nontrivial $i^{th} L^2$-Betti number, then $\actdim(\gG) \ge 2i$.
\end{Aconj}

For a nice introduction to $L^2$-cohomology, we refer the reader to \cite{eck00}. If $\Gamma$ is the fundamental group of a closed Riemannian manifold, then the $L^2$-Betti numbers of $\Gamma$ measure the size of the space of square summable harmonic forms on the universal cover. A positive solution to the Action Dimension Conjecture has several nice implications, in particular it implies a conjecture of Hopf-Thurston that the Euler characteristic of a closed, aspherical $2n$-manifold has sign $(-1)^n$.

Though both $L^2$-cohomology and action dimension seem difficult to compute, the Action Dimension Conjecture has been verified for many important classes of groups, such as lattices in Lie groups \cite{bf}, $\Out(F_n)$ \cite{bkk}, mapping class groups \cite{desp}, \cite{m00}, \cite{g91}, and most Artin groups \cite{ados}, \cite{dl}, \cite{dh}. 

We show the conjecture holds for groups that act properly and cocompactly on thick Euclidean buildings. The $L^2$-Betti numbers of these groups are concentrated in the dimension of the building, and we show that their action dimension is greater than or equal to twice this dimension.

\begin{theorem}\label{main1}
If $\Gamma$ is a cocompact lattice in a thick, $n$-dimensional Euclidean building, then $\actdim(\gG) \ge 2n$. If $\gG$ is torsion-free, then $\actdim(\gG) = 2n$.
\end{theorem}

The second statement follows immediately from Stallings' theorem, since the geometric dimension of such $\gG$ will be $n$. Our main application of this theorem is to compute the action dimension of $S$-arithmetic groups. These groups act properly on the product of a symmetric space and a Euclidean building. Bestvina and Feighn showed that lattices in symmetric spaces have action dimension equal to the dimension of the symmetric space \cite{bf}, and it was conjectured in \cite{bkk} that the action dimension of $S$-arithmetic groups was equal to the dimension of the symmetric space plus twice the dimension of the Euclidean building. We confirm the conjecture for $S$-arithmetic groups over number fields in Section \ref{S-arithmetic}.

To prove Theorem \ref{main1}, we use the \emph{obstructor dimension} method introduced by Bestvina, Kapovich, and Kleiner in \cite{bkk}, which involves finding subcomplexes of the visual boundary of our building that are hard to embed into Euclidean space. Surprisingly, this is what \emph{all} of the above computations of action dimension rely on.  For Euclidean buildings, the visual boundary admits the structure of a spherical building, and the following corollary of our method may be of independent interest.

\begin{theorem}\label{maincor}
Let $\mathcal{B}$ be a finite $k$-dimensional, spherical building, and fix a chamber $C$. Let $\Opp(C)$ denote the subcomplex of chambers opposite to $C$. If $H_k(\Opp(C), \mathbb{Z}_2) \ne 0$, then $\mathcal{B}$ does not embed into $\mathbb{R}^{2k}$. 
\end{theorem}

Therefore, the only way to embed these spherical buildings into Euclidean space is by appealing to general position.  This was shown earlier by Tancer and Worwerk \cite{tw14} for type $A_n$ and certain type $B_n$ buildings. In the general case, we do not know of a thick spherical building which does not have a chamber $\Delta$ with $H_k(\Opp(\Delta), \mathbb{Z}_2) \ne 0$.

In \cite{ados}, the action dimension of right-angled Artin groups was computed. The key tool was a computation of the van Kampen obstruction for certain simplicial complexes called  \emph{octahedralizations}. Our computation for spherical buildings relies on finding embedded octahedralizations based on $\Opp(\Delta)$ inside the spherical building.

This paper is structured as follows: In Sections \ref{vk} and \ref{octa}, we review the obstructor methods of \cite{bkk} and the computation of the van Kampen obstruction of the octahedralization. In Section \ref{sb}, we review spherical buildings and compute their van Kampen obstruction. In Section \ref{actbuild}, we show the Action Dimension Conjecture for lattices in Euclidean buildings, and in Section \ref{S-arithmetic}, we compute the action dimension of $S$-arithmetic groups.

\begin{acknowledge}
I thank Boris Okun for many helpful conversations. The author is partially supported by NSF grant DMS-1045119. This material is based upon work supported under NSF grant DMS-1440140 while the author was in residence at the Mathematical Sciences Research Institute in Berkeley, California in the Fall 2016 semester. The author wishes to thank MSRI for hosting him during this time.
\end{acknowledge}

\section{The van Kampen Embedding Obstruction}\label{vk}

The first cohomological obstruction to embedding a simplicial complex in $\mathbb{R}^n$ was developed by van Kampen \cite{vk}. In this section we review this obstruction and describe how a coarsening of this obstruction gives a lower bound for action dimension. 

Let $K$ a simplicial complex, and let $\Delta \subset K \times K$ denote the simplicial neighbourhood of the diagonal: $$\Delta := \{(\sigma,\tau) | \sigma,\tau \in K,  \sigma \cap \tau \ne \emptyset.\}$$ The complement $K \times K - \Delta$ admits a free $\mathbb{Z}_2$-action by flipping the factors. Let  $\mathcal{C}(K) = (K \times K - \Delta)/\mathbb{Z}_2$ be the \emph{simplicial configuration space of $K$}. So, $\mathcal{C}(K)$ is the space of unordered pairs of disjoint simplices of $K$: $$\mathcal{C}(K) := \{\{\sigma,\tau\} | \sigma,\tau \in K, \sigma \cap \tau = \emptyset\}.$$

\begin{definition}
Let $K$ be a $k$-dimensional simplicial complex, and let $f: K \rightarrow \mathbb{R}^n$ be a general position map. The images of two disjoint simplices $\sigma$ and $\tau$ with $\dim \gs + \dim \gt = n$ intersect under $f$ in a finite number of points. The \emph{$\mathbb{Z}_2$-valued van Kampen obstruction} $\vkt^n(K)  \in H^{n}(\mathcal{C}(K), \mathbb{Z}_2)$ is defined by $\vkt^n(\{\sigma,\tau\}) = |f(\sigma) \cap f(\tau)|\mod 2$. 
\end{definition}

One can show that the class of this cocycle does not depend on $f$, which implies that if $\vkt^n(K) \ne 0$, then $K$ does not embed into $\mathbb{R}^n$. In this case, we say $K$ is an \emph{$n$-obstructor}. An easy argument shows that such $K$ actually cannot embed into any contractible $n$-manifold. 

\begin{example}
Suppose $K_{3,3}$ is the Kuratowski graph, i.e the join of $3$ points with $3$ points. There is a class in $H_2(\mathcal{C}(K_{3,3}), \mathbb{Z}_2)$ which consists of all unordered pairs of disjoint simplices. By drawing the graph in the plane and counting intersections, it is easy to see that $\vkt^2(K)$ evaluates nontrivially on this class. 
\end{example}

We will now use the van Kampen obstruction to give a lower bound for the action dimension. This requires the following concept of Bestvina:

\begin{definition}
A \emph{$\mathcal{Z}$-structure} on a group $\Gamma$ is a pair $(\widetilde X, Z)$ of spaces satisfying the following four axioms:

\begin{itemize}

\item $\widetilde X$ is a Euclidean retract.

\item $X = \widetilde X - Z$ admits a covering space action of $\Gamma$ with compact quotient.

\item $Z$ is a $\mathcal{Z}$-set in $\widetilde X$, i.e. there exists a homotopy $\widetilde X \times [0,1] \rightarrow \widetilde X$ such that $H_0$ is the identity and $H_t(X) \subset X$ for all $t > 0$.

\item The collection of translates of a compact set in $X$ forms a null-sequence in $\widetilde X$, i.e. for every open cover $\mathcal{U}$ of $\widetilde X$ all but finitely many translates are contained in a single element of $U$.
\end{itemize}
\end{definition}

A space $Z$ is a \emph{boundary} of $\Gamma$ if there is a $\mathcal{Z}$-structure $(\widetilde X, Z)$ on $\Gamma$. For example, if $\gG$ acts properly and cocompactly on a $CAT(0)$ space $X$, then compactifying $X$ by the visual boundary $\partial X$ gives a $\mathcal{Z}$-structure on $\Gamma$.

\begin{theorem}[\cite{bkk}]\label{bkk}
Suppose $Z$ is a boundary of a group $\Gamma$, and that $K$ is an embedded $n$-obstructor in $Z$. Then $\actdim(\gG) \ge n+2$.
\end{theorem}

Heuristically, if $\Gamma$ admits a $\mathcal{Z}$-structure and acts properly on a contractible $(n+1)$-manifold $M$, then there would be an injective boundary map $Z \rightarrow \partial M$. Since $M$ is contractible, $\partial M$ should be a $n$-sphere. This would contradict $K$ being an $n$-obstructor, and hence prove Theorem \ref{bkk}. Of course, these optimistic statements are false in general, and the proof of Theorem \ref{bkk} requires much more work.

\begin{example}
Let $\Gamma = F_2 \times F_2 $ be the direct product of two finitely generated free groups. $\Gamma$ acts properly and cocompactly on a product of trees, whose visual boundary is a product of Cantor sets. This contains the graph $K_{3,3}$, so by Theorem \ref{bkk}, $\actdim(\gG) = 4$. More generally, Bestvina, Kapovich, and Kleiner showed that the $n$-fold product of free groups has $\actdim = 2n$, using the fact that the $n$-fold join of $3$ points has $\vkt^{2n-2} \ne 0$.
\end{example}

When we compute the action dimension of $S$-arithmetic groups over number fields, we use the slightly more general concept of \emph{obstructor dimension}, denoted $\obdim(\Gamma)$. For simplicity, we only give the definition for type VF groups. We lose no generality since $S$-arithmetic groups over number fields are virtually of finite type. 

Let $K$ be a simplicial complex, and let $\text{Cone}(K) := K \times [0,\infty)/ (K \times 0)$ denote the infinite cone on $K$. Given a triangulation of $\text{Cone}(K)$, we set every edge to have length $1$ and equip $\text{Cone}(K)$ with the induced path metric.

\begin{definition}
Let $X$ be a proper metric space and $K$ a simplicial complex. A map $h: \text{Cone}(K) \rightarrow X$ is \emph{expanding} if for every $\gs$ and $\gt$ in $K$ with $\sigma \cap \tau = \emptyset$ the images $\text{Cone}(\gs)$ and $\text{Cone}(\gt)$ \emph{diverge}, i.e. for every $D > 0$ there exists $t \in \mathbb{R}^+$ such that $h(\sigma \times [t,\infty])$ and $h(\tau \times [t,\infty])$ are $> D$ apart in $X$.
\end{definition}

\begin{definition}
Let $\gG$ be a discrete group, and assume $\gG$ acts properly and cocompactly by isometries on a contractible proper metric space $X$. Then $\obdim(\gG)$ is the maximal $n+2$ such that there is a proper expanding map $h: \text{Cone}(K) \rightarrow X$ with $K$ an $n$-obstructor.
\end{definition}

The following is the main theorem of \cite{bkk}.

\begin{theorem}\label{bkkmain}
$\obdim(\gG) \le \actdim(\gG)$.
\end{theorem}

Finally, we need the following product lemma for obstructor dimension, which follows immediately from the Join Lemma in \cite{bkk}.

\begin{lemma}\label{product}
Suppose $X$ is a proper cocompact contractible $\gG$-complex and $f: \text{Cone}(J) \times \text{Cone}(K) \rightarrow X$ a proper expanding map. If $J$ is an $n$-obstructor and $K$ is an $m$-obstructor, then $\obdim(\gG) \ge n+m+2$.
\end{lemma}

\section{Octahedralizations}\label{octa}
In this section, we recall the definiition of a certain simplicial complex with nontrivial van Kampen obstruction. This complex was used in \cite{ados} to give lower bounds for the action dimension of right-angled Artin groups. We will show this complex also provides lower bounds for the action dimension of groups acting properly and cocompactly on Euclidean buildings.

Given a finite set $V$, let $\gD(V)$ denote the full simplex on $V$ and let $O(V)$ denote the boundary complex of the octahedron on $V$.
In other words, $O(V)$ is the simplicial complex with vertex set $V\times \{\pm 1\}$ such that a subset $\{(v_0, \geps_0),\dots, (v_k, \geps_k)\}$ of $V\times \{\pm1\}$ spans a $k$-simplex if and only if its first coordinates $v_0,\dots, v_k$ are distinct.
Projection onto the first factor $V\times \{\pm1\}\to V$ induces a simplicial projection $p:O(V)\to \gD(V)$.
We will denote the vertex $(v,+1)$ or $(v,-1)$ by $v^+$ and $v$ respectively, and the simplex $(\sigma,1)$ or $(\sigma,-1)$ by $\sigma^+$ and $\sigma$ respectively, 

Any finite simplicial complex $L$ with vertex set $V$ is a subcomplex of $\gD(V)$.
The \emph{octahedralization} $\OL$ of $L$ is the inverse image of $L$ in $O(V)$:
\[
	\OL:= p\minus (L) \subset O(V).
\]
We also will say that $\OL$ is the result of ``doubling the vertices of $L$''. In particular, an $n$-simplex in $L$ becomes an $n$-octahedron (the $n$-fold join of two points) in $\OL$, and inclusions of simplices induce inclusions of octahedra in a canonical way. We will usually assume that $L$ is a flag complex, which means that if the $1$-skeleton of a simplex is in $L$, then the entire simplex is in $L$.

Fix a simplex $\Delta$ in $L$, and let $D_\Delta(L)$ denote the full subcomplex of $\OL$ containing $L$ and $\Delta^+$. We say $D_\Delta(L)$ is $L$ doubled over $\Delta$.

In \cite{ados}, the van Kampen obstruction of $D_\Delta(L)$ was calculated.

\begin{theorem}\label{ados}
Let $L$ be a $k$-dimensional flag simplicial complex.  We have $\vkt^{2k}(D_\Delta(L)) \ne 0$ if and only if $H_k(L, \mathbb{Z}_2) \ne 0$.
\end{theorem}

\begin{figure}
\begin{center}
		\begin{tikzpicture}[scale = 1]

\node[draw,minimum size=3cm,thick,black,regular polygon,regular polygon sides=8] (a) {};
\node at (0,-1.1) {$\Delta$};
\node at (.1,-3.25) {$\Delta^+$};
 		\foreach \x in {1,2,...,8}{
  \fill (a.corner \x) circle[radius=2pt];	
  
  \draw[ black] (-.5,-3) to (.5,-3);
    \fill (-.5,-3) circle[radius=2pt];
     \fill (.5,-3) circle[radius=2pt];
  
     \draw[ black] (.5,-3) to (a.corner 5);
             \draw[ black] (.5,-3) to (a.corner 7);
       \draw[ black] (-.5,-3) to (a.corner 4);
         \draw[ black] (-.5,-3) to (a.corner 6);

  }

\end{tikzpicture}
\caption{If $L$ is an $n$-cycle, then $D_\Delta(L)$ is a subdivided $K_{3,3}$.}
\end{center}

\end{figure}
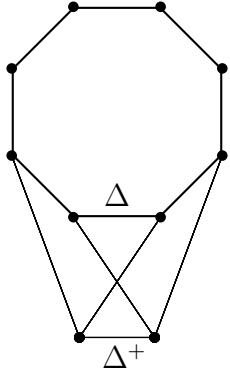

\section{Spherical Buildings}\label{sb}

We will only give a brief description of buildings, the reader is encouraged to refer to \cite{ab07} for complete details. 

\begin{definition}
A \emph{Coxeter group} $(W,S)$ is a group generated by involutions $s_i \in S$, with the only other relations being that every pair of elements $s_is_j$ generates a dihedral group (perhaps $D_\infty$). In other words, $(W,S)$ has a presentation $$<s_1, s_2 , \dots, s_n | (s_is_j)^{m_{ij}} = 1 >$$ for $m_{ij} \in \mathbb{N} \cup \infty$ such that $m_{ii} = 1$.
\end{definition}

For example, proper cocompact reflection groups acting on $\mathbb{S}^n, \mathbb{E}^n$ and $\mathbb{H}^n$ are all examples of Coxeter groups. We will assume from now on that $|S| < \infty$. 

A \emph{mirror structure over $S$} on a space $X$ is a family of subspaces $\{X_s\}_{s\in S}$ indexed by $S$. If $(W,S)$ is a Coxeter system and $X$ has a mirror structure over $S$, let $S(x):=\{s\in S\mid x\in X_s\}$. Now, define an equivalence relation $\sim$ on $W\times X$ by $(w,x)\sim (w',y)$ if and only if $x=y$ and $w^{-1}w'\in W_{S(x)}$. Let $\mathcal{U}(W,X)$ denote the quotient space: $$\mathcal{U}(W,X)=(W\times X)/ \sim.$$
$\mathcal{U}(W,X)$ is called the \emph{basic construction}. There is a natural $W$-action on $W\times X$ which respects the equivalence relation, and hence descends to an action on $\mathcal{U}(W,X)$. One can think of  $\mathcal{U}(W,X)$ as pasting copies of $X$ with the exact gluing given by the Coxeter group.

\begin{definition}
Let $(W,S)$ be a Coxeter group, and let $\Delta$ be a simplex of dimension $|S| - 1$. We can label the codimension-one faces of $\Delta$ by the elements of $S$. The set of codimension-one faces is a mirror structure on $\Delta$. The space $\mathcal{U}(W,S)$ is called the Coxeter complex of $(W,S)$. If $W$ is finite, $\Phi_W$ is homeomorphic to a sphere, and if $W$ is infinite, $\Phi_W$ is contractible. 
\end{definition}

 \begin{definition}A \emph{building} is a simplicial complex with a distinguished set of subcomplexes called \emph{apartments}. This collection of apartments satisfies the following axioms: 
 
 \begin{itemize}
 \item Each apartment is isomorphic to a Coxeter complex.
 \item For any two simplices, there is an apartment containing both of them.
 \item If two apartments contain simplices $\sigma$ and $\tau$, then there is an isomorphism between these apartments fixing $\sigma \cap \tau$ pointwise.
 \end{itemize} 
 \end{definition}
 
The top dimensional simplices of a building are called \emph{chambers}. A building is \emph{thick} if each codimension-one face is contained in at least $3$ chambers. 
 
If $W$ is finite, then the building is called \emph{spherical}, since each apartment is homeomorphic to a sphere. In a spherical Coxeter complex, there is a well-defined notion of \emph{opposite} chambers. Note that every chamber $\sigma$ corresponds to a Coxeter group element $w_\sigma$. Two chambers $\sigma$ and $\tau$ are opposite if they correspond to Coxeter group elements $w_\sigma, w_\tau$ such that $w_\sigma = w_0w_\tau$, where $w_0$ is the longest element of $W$ with respect to the standard generating set $S$. Opposition naturally extends to spherical buildings; in fact, two chambers $\sigma$ and $\tau$ are opposite in one apartment if and only if they are opposite in each apartment containing them both. 

It is important that two chambers in an apartment are opposite if and only if they are on opposite sides of every wall in that apartment, where a wall is a fixed point set inside an apartment of a reflection in $W$. One can analagously define opposition of simplices: $\alpha$ and $\beta$ are opposite in an apartment if they are contained in the same walls and separated by every wall that does not contain them both, and opposite in a building if they are opposite in each apartment which contains them.

 A spherical building is called \emph{right-angled} if all apartments are isomorphic to an $n$-octahedron, which is the Coxeter complex of the right-angled Coxeter group $W \cong \mathbb{Z}_2^{n+1}$.
In an octahedron, the vertices $v$ and $v^+$ are opposite. More generally, two simplices $\sigma$ and $\tau$ in an octahedron are opposite if and only if $p(\sigma) = p(\tau)$ and $\sigma \cap \tau = \emptyset$.

\begin{example}
Let $\mathcal{P}$ denote the poset of subspaces of $\mathbb{R}^n$. Let Flag($\mathcal{P}$) denote the associated flag complex, i.e.  the vertices of Flag($\mathcal{P}$) consist of elements of $\mathcal{P}$ and simplices of Flag($\mathcal{P}$) correspond to chains of inclusions between these subspaces. Flag($\mathcal{P}$) has a natural structure of a spherical building, where apartments correspond to choosing a basis of $\mathbb{R}^n$ and then considering all subspaces generated by proper subsets of these basis vectors.  In this case, the Coxeter group is the symmetric group $S_n$, which acts by permuting the basis vectors. The Coxeter complex can be identified with the barycentric subdivision of the $(n-2)$-simplex. A chamber  is a maximal chain $[v_0] \subset [v_0,v_1] \subset \dots \subset [v_0, v_1, \dots, v_{n-2}]$. Two chambers $C$ and $C'$ are opposite in this building precisely when  any two subspaces in $C$ and $C'$ respectively are in general position. 
\end{example}

\subsection{Bending homeomorphisms}

We will now give our main construction. Essentially, we will exhibit obstructors in a spherical building $\mathcal{B}$ by embedding $D_\Delta(\Opp)$ into $\mathcal{B}$ and applying Theorem \ref{ados}.

\begin{definition}

Let $\mathcal{B}$ be an $n$-dimensional spherical building. Let $A$ be an apartment in $\mathcal{B}$, fix a chamber $C$ in $A$, and let $\{H_1, H_2, ... , H_{n+1}\}$ be the set of walls of $A$ that intersect $C$ in a codimension-one face. If $O^n$ is an $n$-octahedron, then analogously we fix a $n$-simplex $C'$ in 
$O^n$ and define a set of walls $\{H_1', H_2', ... , H_{n+1}'\}$ that intersect $C'$ in a codimension-one face. Then there is an obvious "bending" homeomorphism $f: O^n \rightarrow A$ which maps each wall $H_i' \rightarrow H_i$, maps $C'$ to $C$, and respects all intersections of walls.
\end{definition}

Now, fix a chamber $\Delta^+ \in \mathcal{B}$, and let $\Opp(\Delta^+)$ be the simplicial complex of opposite simplices.  
For each chamber $\sigma \in \Opp(\Delta^+)$, there is a unique apartment $A_\sigma$ containing $\Delta^+$ and $\sigma$. Given an octahedron $O(\sigma)$ which is doubled over $\sigma$, we have a bending homeomorphism $f_\sigma: O(\sigma) \rightarrow A_\sigma$ which takes $\sigma$ to $\sigma$ and $\sigma^+$ to $\Delta^+$. 

\begin{lemma}\label{map}
Suppose $\alpha \subset \sigma$ and let $f_\sigma: O(\sigma) \rightarrow A_\sigma$ be the bending homeomorphism. Then $f_\sigma$ sends $O(\alpha)$ to the convex hull of $\alpha$ and the opposite simplex of $\alpha$ in $\Delta^+$.
\end{lemma}

\begin{proof}

Note that $O(\alpha)$ is the intersection of all walls in $O(\sigma)$ which contain $\alpha$.  
Since $f_\sigma$ preserves the intersection of walls, we have that $f_\sigma$ sends $O(\alpha)$ to the intersection of all walls which contain $f(\alpha)$, and that $f(\alpha^+)$ is opposite to $f(\alpha)$. 
The convex hull in $A_\sigma$ of two simplices is the intersection of all roots which contain them both. In this case, this is the intersection of all walls which contain $f(\alpha)$.
\end{proof}

\begin{lemma}
The collection of bending homeomorphisms $f_\sigma$ extends to a map $F: O(\Opp(\Delta^+)) \rightarrow \mathcal{B}$.
\end{lemma}

\begin{proof}
It suffices to show that if $\sigma$ and $\tau$ are chambers in $\Opp(\Delta^+)$, then $f_\sigma(O(\sigma \cap \tau)) = f_\tau(O(\sigma \cap \tau))$.  By Lemma \ref{map} we know $f_\sigma$ and $f_\tau$ send $O(\sigma \cap \tau)$ to the convex hull of $\sigma \cap \tau$ and $(\sigma \cap \tau)^+$ in $A_\sigma$ and $A_\tau$, so we must show these intersections coincide. Since each apartment is convex, we have that both convex hulls are the convex hull of $\sigma \cap \tau$ and $(\sigma \cap \tau)^+$ in $\mathcal{B}$. Since this is unique, $f_\sigma$ and $f_\tau$ agree on $O(\sigma \cap \tau)$.
\end{proof}

Now, we will show that $F$ restricts to an embedding on $D_\Delta(\Opp(\Delta^+))$, for some choice of a chamber $\Delta$ in $\Opp(\Delta^+)$. To do this, we need to be more precise about the image of simplices under $F$.

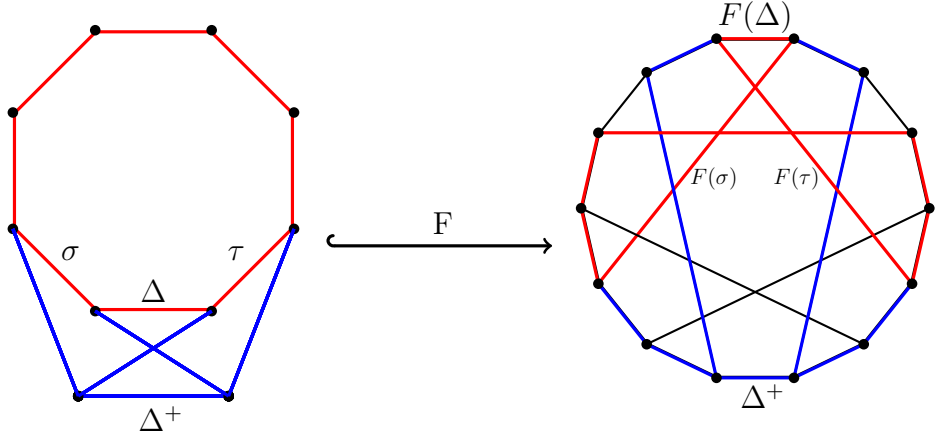
\begin{figure}
\begin{center}
		\begin{tikzpicture}[scale = 1]

 \begin{scope}[shift={(-1,0)}]

\node[draw,minimum size=4cm,very thick,red,regular polygon,regular polygon sides=8] (a) {};

\node at (0,-1.6) {$\Delta$};
\node at (1.1,-1.1) {$\tau$};
\node at (-1.1,-1.1) {$\sigma$};

\node at (.1,-3.3) {$\Delta^+$};
 		\foreach \x in {1,2,...,8}{
  \fill (a.corner \x) circle[radius=2pt];	
  
  \draw[very thick,blue] (-1,-3) to (1,-3);
    \fill (-1,-3) circle[radius=2pt];
     \fill (1,-3) circle[radius=2pt];
       \draw[very thick,blue] (-1,-3) to (a.corner 4);
         \draw[very thick,blue] (-1,-3) to (a.corner 6);
           \draw[very thick,blue] (1,-3) to (a.corner 5);
             \draw[very thick,blue] (1,-3) to (a.corner 7);
  }

\end{scope}

\draw [very thick, right hook->] (1.3,-1) -- (4.3,-1);

\node at (2.85, -.7) {F};
		
 \begin{scope}[shift={(7,-.5)}]
\node[draw,minimum size=4.6cm,thick,regular polygon,regular polygon sides=14] (a) {};

  \draw[very thick,red] (a.corner 1) to (a.corner 6);
 \draw[thick] (a.corner 3) to (a.corner 8);
  \draw[thick] (a.corner 5) to (a.corner 10);
   \draw[thick] (a.corner 7) to (a.corner 12);
    \draw[thick] (a.corner 9) to (a.corner 14);
        \draw[very thick,red] (a.corner 1) to (a.corner 2);
            \draw[very thick,red] (a.corner 4) to (a.corner 5);
                \draw[very thick,red] (a.corner 5) to (a.corner 6);
                    \draw[very thick,red] (a.corner 11) to (a.corner 12);
                    \draw[very thick,red] (a.corner 12) to (a.corner 13);

                     \draw[very thick,blue] (a.corner 14) to (a.corner 1);
                        \draw[very thick,blue] (a.corner 14) to (a.corner 9);
                           \draw[very thick,blue] (a.corner 2) to (a.corner 3);
                              \draw[very thick,blue] (a.corner 3) to (a.corner 8);
                                 \draw[very thick,blue] (a.corner 8) to (a.corner 7);
                                   \draw[very thick,blue] (a.corner 7) to (a.corner 6);
                       \draw[very thick,blue] (a.corner 10) to (a.corner 9);
                                   \draw[very thick,blue] (a.corner 11) to (a.corner 10);
     \draw[very thick,red] (a.corner 11) to (a.corner 2);
       \draw[very thick,red] (a.corner 13) to (a.corner 4);

\draw[very thick,blue] (a.corner 8) to (a.corner 9);
\node at (0,2.55) {$F(\Delta)$};
\node [scale = .7] at (.55,.4) {$F(\tau)$};
\node  [scale = .7] at (-.55,.4) {$F(\sigma)$};

\node at (0.1,-2.5) {$\Delta^+$};
 		\foreach \x in {1,2,...,14}{
  \fill (a.corner \x) circle[radius=2pt];	
  }
	\end{scope}

  \end{tikzpicture}
  \end{center}
 \caption{Embedding $D_\Delta(\Opp(\Delta^+))$ into the flag complex of subspaces of $\mathbb{F}_2^3$. The red cycle is the opposite complex of $\Delta^+$. }
\end{figure}

\begin{definition}
Let $(W,S)$ be a Coxeter group, and let $w \in W$. Then $$\text{In}(w) = \{s \in S | \hspace{.5mm} \ell(ws) < \ell(w)\}$$ where $\ell(w)$ denotes the length of $w$ in $W$ with respect to the generating set $S$.
\end{definition}

 $\text{In}(w)$ is precisely the set of letters with which a reduced expression for $w$ can end. Note that $\text{In}(w^{-1})$ is precisely the elements $s \in S$ such that $\ell(sw) < \ell(w)$, i.e. it consists of $w$ which are separated from the identity chamber by the $s$-wall. 

\begin{definition}
Let $\alpha \in D_\Delta(\Opp(\Delta^+))$ be a simplex. Let $V^+(\alpha)$ denote the vertices of $\alpha$ contained in $\Delta^+$ and $V(\alpha)$ denote the vertices of $\alpha$ contained in $\Opp(\Delta^+)$.
\end{definition}

For a chamber $\sigma$  in $\Opp(\Delta^+)$, identify $O(\sigma)$ as the Coxeter complex of $(\mathbb{Z}_2)^{n+1}$, and identify $\Delta^+$ with the identity chamber. Given a vertex $s \in \sigma$, we say that the $s$-wall is the unique wall in $O(\sigma)$ containing every other vertex in $\sigma$. 

A chamber $\alpha \subset O(\sigma)$ corresponds to an element $w$ with $\text{In}(w^{-1}) = V(\alpha)$, i.e. $w$ is separated from the identity by each $s$-wall for $s \in V(\sigma)$. Since the bending homeomorphism $f$ preserves walls containing $\sigma$ and the chamber $\Delta^+$, we have in general:

\begin{lemma}
For $\sigma \in \Opp(\Delta^+)$ and $\alpha$ a chamber in $O(\sigma)$, $F(\alpha)$ is a union of chambers of $A_\sigma$ corresponding to Coxeter group elements $w$ with $\text{In}(w^{-1}) = V(\alpha)$.
\end{lemma}

\begin{lemma}\label{embed}
$F$ restricted to $D_\Delta(\Opp(\Delta^+))$ is an embedding.
\end{lemma}

\begin{proof}
It is obvious that $F$ restricted to each simplex is a homeomorphism. Therefore, it suffices to show that disjoint top dimensional simplices are mapped disjointly under $F$. Note that this is obviously true for $F$ restricted to $\Opp(\Delta^+)$. 
Let $\sigma,\tau$ be chambers in $\Opp(\Delta^+)$, and let $\alpha$ and $\beta$ be disjoint chambers in $O(\sigma) \cap D_\Delta(\Opp(\Delta^+))$ and $O(\tau) \cap D_\Delta(\Opp(\Delta^+)$ respectively.  We must show that $F(\alpha) \cap F(\beta) = \emptyset$. Again, we identify $\Delta^+$ with the identity chamber of $W$. 

If $\alpha$ contains a vertex $s^+$ in $\Delta^+$, the chambers in the image of $F(\alpha)$ correspond to Coxeter group elements $w$ with $s \notin \text{In}(w^{-1})$. On the other hand, since $s^+ \notin \beta$, the chambers in $F(\beta)$ correspond to elements with $s \in \text{In}(w^{-1})$. Therefore, any intersection between $\alpha$ and $\beta$ must occur on the $s$-wall of the respective apartments. 

If $\alpha$ contains no other vertices of $\Delta^+$, then the intersection of the image of $\alpha$ with the $s$-wall is in $\Opp(\Delta^+)$. Therefore, if there is a nontrivial intersection between $\alpha$ and $\beta$, $\alpha$ must contain another vertex in $\Delta^+$. Repeating this argument verbatim would eventually imply that $\alpha = \Delta^+$, which implies $F(\alpha) \cap F(\beta) = \emptyset$. \end{proof}

Now that we have embedded $D_\Delta(\Opp(\Delta^+))$ into our spherical building, we would like to verify the assumptions of Theorem \ref{ados}. Therefore, we need to show that $\Opp(\Delta^+)$ is a flag complex with top dimensional $\mathbb{Z}_2$-homology. 
The first statement follows straight from the definitions, we record it as a lemma.

\begin{lemma}\label{flag}
For any chamber $C$ in $\mathcal{B}$, $\Opp(C)$ is a flag complex.
\end{lemma}

\begin{proof}
We will show that $\Opp(C)$ is a full subcomplex. This means that if the vertex set of a simplex in $\mathcal{B}$ is contained in $\Opp(C)$, then the simplex itself is contained in $\Opp(C)$. It implies that $\Opp(C)$ is also a flag complex. 

We prove the contrapositive: suppose $\sigma$ is a simplex not in $\Opp(C)$, so there exists a wall such that $\sigma$ and $C$ lie on the same side. Therefore, all the vertices of $\sigma$ are not opposite to $C$. Therefore, if a simplex has vertex set in $C$, then the simplex is in $C$.
\end{proof}

On the other hand, it is not obvious to us that $H_k(\Opp(C), \mathbb{Z}_2) \ne 0$ for every thick spherical building. In the next lemma, we verify that this does occur for infinitely thick buildings, which will suffice for our purposes. 

\begin{lemma}\label{thick}
If $\mathcal{B}$ has infinite thickness, then there exists a chamber $C$ such that $\Opp(C)$ contains an apartment.
\end{lemma}

\begin{proof}
Let $A \subset \mathcal{B}$ be any apartment and $\Delta_0$ be any chamber. If $\Delta_0$ is opposite to all chambers in $A$, then we are done. If not, for each chamber $A_i$ in $A$, there are only finitely many chambers that share a panel with $\Delta_0$ and are closer to $A_i$. Therefore, we can `push' $\Delta_0$ to a new chamber $\Delta_1$ that is further from each $A_i$ (or perhaps the same distance if $\Delta_0$ is already opposite to $A_i$). Continuing in this way, we find a chamber $C_0$ opposite to each $A_i$, and therefore $A_i \in \Opp(C_0)$. 
\end{proof}

\begin{example}
If $k$ is an infinite field and Flag($\mathcal{P}$) the associated spherical building, then it is easy to construct an apartment in $\Opp(C)$ for any chamber $C$. Any chamber $C$ corresponds to a flag of subspaces $C = [v_0] \subset [v_0,v_1] \subset \dots \subset [v_0, v_1, \dots, v_{n-2}]$. The opposite apartment is determined by choosing another basis in general position (e.g. choose all $e_k$ to not lie in $[v_0, v_1 \dots, v_{n-2}]$).\end{example}

\begin{remark}

For the finite thickness case, there are some obvious cases where we can find cycles in $\Opp$. If the thickness of each panel is odd, then $\Opp$ is itself a $\mathbb{Z}_2$-cycle. Note that this implies that $\Opp$ has nontrivial homology if the spherical building contains a spherical subbuilding of odd thickness. Also, for large enough thickness, a simple counting argument shows that $\Opp$ has top homology as it contains more $n$-simplices than $(n-1)$ simplices. 

In general, we do not know if there always exist chambers such that $\Opp$ has top dimensional homology. This may be subtle: in \cite{a}, Abramenko constructs examples of infinitely thick $1$-dimensional buildings where for certain chambers $\Opp$ is a disjoint union of trees. However, other chambers inside these buildings have opposite complexes which contain apartments by Lemma \ref{thick}. Note that higher rank spherical buildings have been fully classified by Tits \cite{tits}. It seems likely that one can construct top dimensional cycles in the opposite complexes for each list in the classification, but the calculations were too hard for the author.
\end{remark}

\section{The Action Dimension Conjecture for lattices in Euclidean buildings}\label{actbuild}

The following theorem is classical. In \cite{ddjo}, an alternate proof is given which also computes the $L^2$-cohomology of arbitrary buildings in terms of the weighted $L^2$-cohomology of an apartment.

\begin{theorem}
If $\Gamma$ acts properly and cocompactly on a thick, $n$-dimensional, Euclidean building, then the $L^2$-cohomology of $\gG$ is concentrated in dimension $n$, and is nontrivial if the building is thick. 
\end{theorem}

Therefore, the following is a confirmation of the Action Dimension Conjecture in this case.

\begin{theorem}\label{act2n}
If $\Gamma$ acts properly and cocompactly on a thick, Euclidean building of rank $n$, then $\actdim(\Gamma) \ge 2n$. If $\gG$ is torsion-free, $\actdim(\gG) = 2n$.
\end{theorem}

\begin{proof}
It is well known that the visual boundary of a thick, Euclidean building admits the structure of a thick, spherical building of rank $n-1$ with infinite thickness [Section 11.8, \cite{ab07}]. By Theorem \ref{bkk}, Theorem \ref{ados}, and Lemmas \ref{embed}--\ref{thick}, it follows that $\actdim(\Gamma) \ge 2n$. If $\gG$ is torsion-free, then equality follows from Stallings' theorem since the geometric dimension of $\gG$ is $n$.
\end{proof}

\subsection{More general buildings}

For general buildings, the above strategy fails as we lose the spherical building at infinity. However, there are some specific examples where the action dimension is known. For example, Dymara and Osajda have shown that the boundary of a thick, $n$-dimensional right-angled hyperbolic building is the universal Menger space \cite{do} (these exist for $n \le 4$).  For non right-angled buildings, this was shown for $n = 2$ by Benakli \cite{b}. Therefore, for cocompact lattices acting on such buildings the action dimension is at least twice the dimension of the building. It would be nice to extend Theorem \ref{act2n} to all hyperbolic buildings, where we suspect the same bounds on action dimension should hold. The link of each vertex in a thick, locally finite, hyperbolic building $\mathcal{B}$ is a thick, finite, spherical building. The difficulty here is pushing this link in a compatible way to the building to get a proper expanding map $\text{Cone}(\lk) \rightarrow \mathcal{B}$.

For general buildings, the geometric dimension of lattices can be less than the dimension of the building, so by Stallings' theorem, the action dimension must sometimes be less than twice the dimension of the building. For example, let $L$ be a flag $(n-1)$-dimensional flag complex such that $H_{n-1}(L, \mathbb{Z}_2)  = 0$. Form any graph product $G_L$ such that every vertex group $G_s$ is finite. Then $G_L$ acts properly and cocompactly on a $n$-dimensional thick, locally finite right-angled building $X_L$ \cite{dm}. The geometric dimension of $G_L$ is the same as the geometric dimension of the underlying Coxeter group, which has been explicitly computed and with these assumptions is less than $n$ \cite{ddjmo}.  We conjecture that if $\gG$ acts properly and cocompactly on a locally finite, thick building then $\actdim(\gG) \ge 2\gdim(\gG)$.

\section{$S$-arithmetic groups}\label{S-arithmetic}

In this section, we describe our main application of Theorem \ref{act2n}.  Let $k$ be a number field, i.e. a finite degree extension of $\mathbb{Q}$. Let $G(k)$ be a semisimple linear algebraic group over $k$.  For each finite place $p$ let $\nu_p: k \rightarrow \mathbb{Z}$ denote the corresponding valuation. Let $S$ be a set of finite places of $k$, including the set of infinite places. Let $k_p$ be the completion of $k$ with respect to the norm induced by $\nu_p$. The ring of $S$-integers is defined by: $$\mathcal{O}_S := \{x \in k | \nu_p(x) \ge 0 \hspace{1mm} \text{for all} \hspace{1mm} p \notin S\}.$$ A subgroup of $G(k)$ is an $S$-arithmetic subgroup if it is commensurable with $G(\mathcal{O}_S)$.

\begin{example}
Let $S = \{p_1, p_2, \dots p_m\}$ be a finite set of primes. Each prime determines a $p$-adic valuation $\mathbb{Q} \rightarrow \mathbb{Z}$ which sends $u \rightarrow n$ if $u = p^nx$ and neither the numerator nor the denominator of $x$ is divisible by $p$. The ring of $S$-integers in this case is $\mathbb{Z}_S := \mathbb{Z}[\frac{1}{p_1}, \frac{1}{p_2}, \dots \frac{1}{p_m}]$. An \emph{$S$-arithmetic subgroup} $\Gamma$ of $G(\mathbb{Q})$ is commensurable to $G(\mathbb{Z}_S)$.
\end{example}

For simplicity, we will pass to a torsion-free finite index subgroup $\gG$ of $G(\mathcal{O}_S)$, which always exists in this setting \cite{bs}.
If $S$ consists of only infinite places, then $\Gamma$ is an arithmetic subgroup and acts on a symmetric space $X_\infty$. 
If $S$ has finite places, then for each finite place $\nu$  there is a Euclidean building $X_\nu$ associated to $G(k_\nu)$, and $\Gamma$ acts properly on $X_\infty \times \prod_{\nu_p} X_{\nu_p} $. Furthermore, there is a partial compactification of $X_\infty$ due to Borel-Serre, which we denote $X_\infty^{BS}$. The $\gG$-action extends to a proper and cocompact action on $X_\infty^{BS} \times \prod_{\nu_p} X_{\nu_p}$ \cite{bs}. It will be important for us that if we fix a vertex $v$ in $\prod_{\nu_p} X_{\nu_p}$ and restrict to the $\gG$-action on  $\prod_{\nu_p} X_{\nu_p}$, the stabilizer of $v$ in $\gG$ is an arithmetic subgroup $\gG_\infty$.

Bestvina and Feighn construct a proper, expanding map $f$ from a coned $(\dim(X_\infty) - 2)$-obstructor $\text{Cone}(J)$ into $X_\infty$ that is bounded distance from a $\gG_\infty$-orbit.
Theorem \ref{act2n} gives a proper, expanding map $g$ from a coned $(2\dim(\prod_{\nu_p} X_{\nu_p}) -2)$-obstructor $\text{Cone}(K)$ into $\prod_{\nu_p} X_{\nu_p}$.  

We take the product embedding into $X_\infty \times \prod_{\nu_p} X_{\nu_p}$ and then compose with the inclusion into the Borel-Serre partial compactification: $$h:= f \times g: \text{Cone}(J) \times \text{Cone}(K) \rightarrow  X_\infty^{BS} \times \prod_{\nu_p} X_{\nu_p}$$.

We will now show that $h$ is proper and expanding with respect to a proper $\gG$-invariant metric on $X_\infty^{BS} \times \prod_{\nu_p} X_{\nu_p}$. We first show that such a metric exists. Since $X_\infty^{BS}$ is a manifold with corners, it is separable and metrizable, and hence so is $X_\infty^{BS} \times \prod_{\nu_p} X_{\nu_p}$. Therefore, the following general theorem applies.

\begin{theorem}[Theorem 4.2, \cite{ans}]\label{metric}
Suppose $X$ is $\sigma$-compact. If the action of $\gG$ on $X$ is proper and $X$ is metrizable, then there is a $\gG$-invariant proper metric $d_\gG$ on $X$ that induces the topology of $X$.
\end{theorem}

Suppose $\gG$ acts by isometries on two metric spaces $X$ and $Y$, and suppose the diagonal action of $\gG$ on $X \times Y$ is proper. Suppose that $X$ admits a partial compactification $\hat X$ such that the $\gG$-action extends to a proper and cocompact action on $\hat X \times Y$. 
Let $d_X$ and $d_Y$ denote the left $\gG$-invariant metrics on $X$ and $Y$, and let $d_\gG$ be a left $\gG$-invariant metric on $\hat X \times Y$ as in Theorem \ref{metric}. We shall prove the following general theorem in the next section.

\begin{theorem}\label{S}
Suppose that there is a proper expanding map $f: \text{Cone}(J) \rightarrow X$ and a proper expanding map $g: \text{Cone}(K) \rightarrow Y$. Let $(x,y) = (f(0),g(0))$, and suppose $\text{Cone}(J)$ maps a bounded distance in the $d_X$-metric from the $\text{Stab}_\gG(y)$-orbit of $(x,y)$. For any $\gG$-invariant proper metric on $\hat X \times Y$, the product map $h = f \times g: \text{Cone}(J) \times \text{Cone}(K) \rightarrow \hat X \times Y$ is proper and expanding.
\end{theorem}

By the above remarks and Theorem \ref{bkkmain} this theorem implies the following:
\begin{corollary}\label{example}
If $\Gamma$ is an $S$-arithmetic group over a number field then $h: \text{Cone}(J) \times \text{Cone}(K) \rightarrow  X_\infty^{BS} \times \prod_{\nu_p} X_{\nu_p}$ defined as above is proper and expanding, and hence $\actdim(\Gamma) \ge \dim(X_\infty) + \sum_{i = 1}^{|S|} 2\dim(X_\nu).$ 
\end{corollary}

\subsection{Proof of Theorem \ref{S}}
We first need the following two lemmas. 

\begin{lemma}\label{boundedy}
Let $(x_i,y_i)$ and $(x_i',y_i')$ in $\hat X \times Y$ be a pair of sequences.  
If $d_Y(y_i, y_i') \rightarrow \infty$, then $d_\gG((x_i,y_i), (x_i',y_i')) \rightarrow \infty$.
\end{lemma}

\begin{proof}
By contradiction, suppose there exists a subsequence $(x_i,y_i), (x_i',y_i')$ with $d_\gG((x_i,y_i), (x_i',y_i'))$ uniformly bounded.
Since the $\gG$-action on $\hat X \times Y$ is cocompact, there exists $\gamma_i \in \gG$ such that $(\gamma_ix_i,\gamma_iy_i)$, and hence $(\gamma_ix_i',\gamma_iy_i')$ are contained in a compact set $K$. We still have $d_Y(\gamma_iy_i,\gamma_iy_i') \rightarrow \infty$ since $d_Y(\gamma y_i, \gamma y_i') = d_Y(y_i, y_i')$. 
Since $K$ projects to a compact set in $Y$, this is a contradiction.
\end{proof}

\begin{lemma}\label{boundedx}
Let $(x_i,y_i)$ and $(x_i',y_i')$ in $X \times Y$ be a pair of sequences.  
Assume that $y_i$ remains bounded $d_Y$-distance from a basepoint $y_0$.
Suppose that $x_i$ and $x_i'$ are bounded $d_X$-distance from the $\text{Stab}_\gG(y_0)$-orbit of a basepoint $x_0$.
If $d_X(x_i, x_i') \rightarrow \infty$, then $d_\gG((x_i,y_i), (x_i',y_i')) \rightarrow \infty$.
\end{lemma}

\begin{proof}
Again, assume by contradiction that there exist subsequences $(x_i,y_i)$ and $(x_i',y_i')$, with $d_\gG((x_i,y_i), (x_i',y_i'))$ uniformly bounded. 
Choose $\gamma_i$ in $\text{Stab}_\gG (y_0)$ such that $d_X(\gamma_ix_i,x_0)$ is uniformly bounded. Since we assume $y_i$ are in a compact set of $Y$, we know that $d_Y(\gamma_iy_i, y_0)$ is uniformly bounded. The sequence $(\gamma_ix_i,\gamma_iy_i)$ is contained in a compact set in $X \times Y$, which implies $d_\gG((\gamma_ix_i,\gamma_iy_i), (x_0,y_0))$ is uniformly bounded.  Therefore by assumption $d_\gG((\gamma_ix_i',\gamma_iy_i'), (x_0,y_0))$ is uniformly bounded.

We have $d_X(\gamma_ix_i',x_0)  \rightarrow \infty$ since $d_X(\gamma_ix_i, \gamma_ix_i') \rightarrow \infty$ and $d_X(\gamma_ix_i,x_0)$ is uniformly bounded.
Choose $\gamma_i'$ in $\text{Stab}_\gG (y_0)$ such that $d_X(\gamma_i'x_i',x_0)$ is uniformly bounded. 
Then the sequence $d_\gG((\gamma_i(\gamma_i')^{-1}x_0, y_0),(x_0,y_0))$ is uniformly bounded
Since $d_X(x_i,x_i') \rightarrow \infty$, we have the distance in the word metric of $\gG$ between $\gamma_i$ and $\gamma_i' \rightarrow \infty$, which contradicts the properness of the $\gG$-action on $\hat X \times Y$.
\end{proof}

Now, we can prove Theorem \ref{S}.

\begin{proof}
We show that $f \times g$ is proper and expanding. We will let $(j,k)$ denote points in $\text{Cone}(J) \times \text{Cone}(K)$.

\emph{Proper:} Let $(f(j_0),g(k_0))$ be the image of the cone point. 
Suppose $f \times g$ is not proper. Then there is a sequence $(j_i,k_i)$ which leaves every compact set of $\text{Cone}(J) \times \text{Cone}(K)$ such that $h(j_i, k_i)$ is contained in a compact set $C$ of $\hat X \times Y$. 
Since $C$ projects to a compact set in $Y$ and $g$ is proper, the $k_i$ are contained in a compact subset of $\text{Cone}(K)$. Therefore, the $j_i$ leave every compact set in $\text{Cone}(J)$, and since $f$ is proper, this implies that $d_X(f(j_0),f(j_i)) \rightarrow \infty$ which contradicts Lemma \ref{boundedx}. 
 
 \emph{Expanding:}  Assume $\gs \times \gt$ and $\gs' \times \gt'$ are disjoint simplices in $\text{Cone}(J) \times \text{Cone}(K)$. Suppose $f \times g$ is not expanding. Then there are sequences $(j_i,k_i)$ in $\text{Cone}(\sigma) \times Cone(\tau)$ and $(j_i',k_i')$ in $\text{Cone}(\sigma') \times \text{Cone}(\tau')$, which leave every compact set, and have $d_\gG(h(j_i,k_i), h(j_i',k_i'))$ uniformly bounded. By Lemma \ref{boundedy}, this implies that $d_Y(k_i,k_i')$ is uniformly bounded. Since $g$ is expanding, this implies that $k_i$ and $k_i'$ are contained in a compact set. Therefore, $j_i$ and $j_i'$ leave every compact set, and since $f$ is expanding we have $d_X(f(j_i), f(j_i')) \rightarrow \infty$, which contradicts Lemma \ref{boundedx}.
\end{proof}

\begin{example}
Let $\gG = SL_2(\mathbb{Z}[1/2])$. In this case, $\gG$ acts properly on $\mathbb{H}^2 \times T$, where $T$ is a trivalent tree. The obstructor complex we want to use is $\text{Cone}(K_{3,3})   \cong \text{Cone}(3\text{pts}) \times \text{Cone}(3\text{pts})$. In the trivalent tree, we map $\text{Cone}(3\text{pts})$ by choosing $3$ disjoint rays emanating from a fixed base point. 
In $\mathbb{H}^2$, which we identify with the upper half plane, we send $\text{Cone}(3\text{pts})$ to the orbit of $i$ under the matrices: \[
   \{
  \left[ {\begin{array}{cc}
   1 & t \\       
   0 & 1
       \end{array} } \right] | \hspace{1mm}t \in \mathbb{R} \} \hspace{1mm} \bigcup \hspace{1mm} \{  \left[ {\begin{array}{cc}
   1 & 0 \\       
   t & 1
       \end{array} } \right] | \hspace{1mm}t \in \mathbb{R}^+
       \}
\]

Unfortunately, one cannot immediately use the product embedding $\text{Cone}(3\text{pts}) \times \text{Cone}(3\text{pts}) \rightarrow \mathbb{H}^2 \times T$.  This map is obviously proper and expanding, and if the image was bounded distance from an $SL_2(\mathbb{Z}[1/2])$-orbit we could take a neighbourhood of the orbit to see that $\obdim(SL_2(\mathbb{Z}[1/2])) = 4$. The difficulty here is that the image is \emph{not} bounded distance from any orbit. It seems in this case that one could perturb the product embedding to stay within a bounded distance of an orbit, but for the general case we chose to keep our simple definition of the product embedding and lose the nice structure of the product metric on $\mathbb{H}^2 \times T$. 

\end{example}

\begin{remark}
These obstructor methods also apply for certain $S$-arithmetic groups over function fields. Let $K$ be the function field of an irreducible projective smooth curve $C$ defined over a finite field $k := \mathbb{F}_q$. Let $S$ be a finite nonempty set of (closed) points of $C$, and let $\mathcal{O}_S < K$ be the ring of functions that have no poles except possibly at points in $S$. If $S$ is a single point $p$, it determines a valuation $\nu_p: K \rightarrow \mathbb{Z}$ which assigns to a function its order of vanishing at $p$. 

If $G$ is a linear algebraic group over $K$, then any group commensurable to $G(\mathcal{O}_S)$ is an $S$-arithmetic subgroup. Similarly as in the number field case, to each $S$-arithmetic group there is a corresponding Euclidean building $X_S$, and in this case $G(\mathcal{O}_S)$ acts properly on $X_S$, with no symmetric space factor. In the general setting, the action is not cocompact, and we cannot conclude anything from Theorem \ref{act2n}. However, we do have the following theorem:
\end{remark}
\begin{theorem}
For $K$ and $S$ as above, if $G(\mathcal{O}_S)$ acts cocompactly on $X_S$, then $\actdim(G(\mathcal{O}_S)) \ge 2\dim X_S$ (this is true exactly when the $K$-rank of $G = 0$.)
\end{theorem}

\end{document}